\documentclass[a4paper,11pt]{amsart}

\usepackage{mathrsfs}
\usepackage{amssymb}
\usepackage{amscd}
\usepackage[all]{xy}
\usepackage{amsmath}
\usepackage{dsfont}
\usepackage{stmaryrd}
\usepackage{currvita}
\usepackage{hyperref}
\usepackage[latin1]{inputenc}
\usepackage{ae}
\usepackage{amsthm}

\usepackage{xypic}

\usepackage{ifthen}
\xyoption{curve}

\newtheorem{lemma}{Lemma}[section]
\newtheorem{prop}[lemma]{Proposition}

\newtheorem{thm}[lemma]{Theorem}
\newtheorem{cor}[lemma]{Corollary}

\theoremstyle{definition}

\def\phi{\varphi}
\def\theta{\vartheta}
\def\epsilon{\varepsilon}

\def\nin{\notin}

\def\iso{\cong}

\newcommand\set[2][auto]{
     \ifthenelse{\equal{#1}{auto}}{\left\lbrace}{\csname #1\endcsname\lbrace} #2 
     \ifthenelse{\equal{#1}{auto}}{\right\rbrace}{\csname #1\endcsname\rbrace} }

\def\mathbb{\mathds}

\newcommand\isoto{\mbox{$\hspace{7.5pt}\raise 3pt\hbox{$\sim$}\hspace{-17pt}\longrightarrow\hspace{3pt}$}\linebreak[0]}
\newcommand\isoot{\mbox{$\hspace{8.5pt}\raise 3pt\hbox{$\sim$}\hspace{-18pt}\longleftarrow\hspace{3pt}$}\linebreak[0]}

\DeclareMathOperator{\preim}{pre\kern0.13em im} 
\DeclareMathOperator{\preker}{pre\kern0.13em ker} 
\DeclareMathOperator{\precoker}{pre\kern0.13em coker} 

\DeclareMathOperator{\Aut}{Aut}

\begin{document}
\title{Monodromy of subvarieties of PEL-Shimura varieties}
\author{Ralf Kasprowitz} 

\address{Universit\"at Paderborn, Fakult\"at f\"ur Elektrotechnik, Informatik und Mathematik,
Institut f\"ur Mathematik, Warburger Str. 100, 33098 Paderborn, Germany}
%\curraddr{}

\email{kasprowi@math.upb.de}

\begin{abstract} The aim of this paper is to generalize results of C.-L. Chai about the monodromy of Hecke invariant subvarieties in \cite{chaimonodromy} to Shimura varieties of PEL-type.   \end{abstract}

\maketitle

%Mathematical Subject Classification (2010): primary: 
%secondary: 

%Keywords: 

\section{Introduction}

This paper deals with a generalization of the following results of C.-L. Chai about Hecke invariant subvarieties. Let $\mathcal{A}_{g,N}$ denote the moduli space over $\overline{\mathbb{F}}_p$ of $g$-dimensional principally polarized abelian varieties in characteristic $p$ with symplectic level-$N$ structure. Let $l \neq p$ be a prime number and let $N >3$ be a natural number relatively prime to $pl$. Let $Z$ be a smooth locally closed subvariety  of $\mathcal{A}_{g,N}$. Assume that $Z$ is stable under all $l$-adic Hecke correspondences coming from $\mathbf{Sp}_{2g}(\mathbb{Q}_l)$ and that the $l$-adic Hecke correspondences operate transitively on the set of connected components of $Z$. Furthermore, let $Z^0$ be a connected component of $Z$ with generic point $z$ and denote by $A \rightarrow Z^0$ the universal abelian scheme restricted to $Z^0$. Then there are the following Propositions:

\begin{prop}{\cite[Proposition 4.1]{chaimonodromy}}
Assume that the $l$-adic monodromy representation $\rho_l:\pi_1(Z^0,\overline{z}) \rightarrow \mathbf{Sp}_{2g}(\mathbb{Q}_l)$ attached to $A \rightarrow Z^0$ has infinite image. Then its image contains an open subgroup of $\mathbf{Sp}_{2g}(\mathbb{Q}_l)$.
\end{prop}
 
One can show that the conditions of this Proposition imply that $Z$ is connected. This can be applied to subvarieties which are not contained in the supersingular locus of $\mathcal{A}_{g,N}$.

\begin{prop}{\cite[Proposition 4.4]{chaimonodromy}}
Assume that $A_{\overline{z}}$ is not supersingular. Then $Z = Z^0$.
\end{prop}

One can use these results for example to show that non-supersingular Newton strata are irreducible (\cite[Theorem 3.1]{chaioort}).
In this paper, we generalize the first result as described in the following Proposition. We denote by $Sh_{K^p_0}$ the geometric special fiber of a moduli space $\mathcal{S}\mathfrak{h}_{K^p_0}$ corresponding to an integral Shimura PEL-datum constructed by R. E. Kottwitz in \cite{kottwitz}. The associated reductive group over $\mathbb{Q}$ is denoted by $\mathbf{G}$, with derived group $\mathbf{G}_1$. Let $P$ be the finite set of primes of $\mathbb{Q}$ containing $p$ and the primes $l$ such that some simple component of $\mathbf{G}_1$ is $\mathbb{Q}_l$-anisotropic.

\begin{prop} 
Let $\mathcal{D}$ be an integral Shimura PEL-datum, unramified at a prime $p$. Let $\mathbf{H}_1, \dots,\mathbf{H}_n$ be the simple components of the derived group $\mathbf{G}_1$. Let $Z \subseteq Sh_{K^p_0}$ be a smooth, locally closed subscheme. \begin{enumerate} \item Suppose that for a prime $l \notin P$ the $l$-Hecke correspondences of the simply connected covering of $\mathbf{G}_1$ act transitively on the set of connected components of $Z$. If for all $i=1,\dots,n$ the set $\mbox{im}(\rho_l) \cap \mathbf{H}_i(\mathbb{Q}_l)$ is not finite, then $\mbox{im}(\rho_l) = K_{0,l}$. \item Suppose that the prime-to-$P$ Hecke correspondences of the simply connected covering of $\mathbf{G}_1$ act transitively on the set of connected components of $Z$. If for all $i=1,\dots,n$ and all $l \notin P$ the set $\mbox{im}(\rho_l) \cap \mathbf{H}_i(\mathbb{Q}_l)$ is not finite, then $\mbox{im}(\rho^P) = K^P_0$. \end{enumerate}
\end{prop}

Finally, we use recent results of E. Viehmann and T. Wedhorn about the Ekedahl-Oort stratification of PEL Shimura varieties to generalize the second Proposition. Here, $\mathcal{N}_{b_0}$ denotes the basic Newton stratum of $Sh_{K^p_0}$.

\begin{thm} Let $\mathcal{D}$ be an integral Shimura PEL-datum of type $A$ or $C$, unramified at a prime $p$.
Suppose that the prime-to-$P$ Hecke correspondences of $\mathbf{G}_1$ act transitively on the set of connected components of $Z$. If $z \nin \mathcal{N}_{b_0}$, then $Z$ is connected.  
\end{thm}

One might hope to use this Theorem to show the irreducibility of non-basic Newton strata for PEL Shimura varieties similar to the Siegel case.

I am very grateful to Eike Lau, Torsten Wedhorn and Daniel Wortmann for many valuable discussions and comments.

\section{Preliminaries}

\subsection{Shimura PEL-data and their moduli problems}
In this section we recall the definition of a Shimura PEL-datum and the associated moduli space of abelian varieties with additional structures. The main references are \cite{kottwitz} and \cite{lan}. \medskip

A tuple $\mathcal{D}=(B,^*,V,\langle\,,\rangle,\mathcal{O}_B,\Lambda,h)$ is called an integral Shimura PEL-datum, unramified at a positive prime $p  \in \mathbb{Z}$, if it consists of the following data:
\begin{itemize} \item $B$ is a finite-dimensional simple $\mathbb{Q}$-algebra, such that there is an isomorphism $B_{\mathbb{Q}_p} \iso \bigoplus_{\mathfrak{p}\in \mathcal{S}}M_d(F_{\mathfrak{p}})$, where $F$ denotes the center of $B$ and $\mathcal{S}$ the set of primes of $F$ over $p$, with $F_{\mathfrak{p}}/\mathbb{Q}_p$ unramified for all $\mathfrak{p} \in \mathcal{S}$. 
 \item 
$^*$ is a $\mathbb{Q}$-linear positive involution on $B$.
\item $V$ is a finitely generated left $B$-module.
\item $\langle\,,\rangle: V \times V \rightarrow \mathbb{Q}$ is a symplectic form on $V$ with $\langle bv,w\rangle=\langle v,b^*w\rangle$ for all $v,w \in V$ and all $b \in B$.
\item $\mathcal{O}_B$ is a $^*$-invariant $\mathbb{Z}_{(p)}$-order of $B$ such that $$\mathcal{O}_B \otimes_{\mathbb{Z}_{(p)}}\mathbb{Z}_p \iso \bigoplus_{\mathfrak{p}\in \mathcal{S}}M_d(\mathcal{O}_{F,\mathfrak{p}}),$$ where $\mathcal{O}_F$ denotes the ring of integers of $F$ and the isomorphism is induced by the corresponding isomorphism for $B$ described above. 
\item $\Lambda$ is an $\mathcal{O}_B$-invariant $\mathbb{Z}_{(p)}$-lattice in $V$, such that $\langle\,,\rangle$ induces a perfect pairing $\Lambda_{\mathbb{Z}_p} \times \Lambda_{\mathbb{Z}_p} \rightarrow \mathbb{Z}_p$. 
\item $h: \mathbb{C} \hookrightarrow \mbox{End}_B(V) \otimes_{\mathbb{Q}}\mathbb{R}$ is a homomorphism such that, if $\iota$ is the involution on $\mbox{End}_B(V)$ coming from $\langle\,,\rangle$, then $h(\overline{z}) = h(z)^{\iota}$ and the form $(v,w) \mapsto \langle v,h(i)w\rangle$ on $V_{\mathbb{R}}$ is positive definite. 
\end{itemize}
There are algebraic $\mathbb{Q}$-groups $\mathbf{G}$ and $\mathbf{U}$ defined by $$\mathbf{G}(R):=\{g \in \mbox{GL}_B(V\otimes_{\mathbb{Q}}R)\;|\; gg^{\iota}\in R^{\times}\}$$ and $$\mathbf{U}(R):=\{g \in \mbox{GL}_B(V\otimes_{\mathbb{Q}}R)\;|\; gg^{\iota} = 1\}.$$ The derived group of $\mathbf{G}$ is denoted by $\mathbf{G}_1$, hence we have an inclusion $$\mathbf{G}_1 \subset \mathbf{U} \subset \mathbf{G}.$$ \medskip

For each Shimura PEL-datum as above there is the following associated moduli problem due to R. E. Kottwitz (\cite[\S5]{kottwitz}). Let $\mathbb{A}^p_f$ denote the ring of finite adeles with trivial $p$-component and let $K^p \subseteq \mathbf{G}(\mathbb{A}^p_f)$ be an open compact subgroup. Furthermore, we denote by $E$ the reflex field associated to the Shimura PEL-datum $\mathcal{D}$ and by $\mathcal{O}_E$ its ring of integers. Consider the functor $$\mathcal{A}:=\mathcal{A}_{\mathcal{D},K^p} = (\mbox{Sch}/\mathcal{O}_{E,(p)}) \rightarrow (\mbox{Sets})$$ from locally noetherian $\mathcal{O}_{E,(p)}$-Schemes to the category of sets that maps an $\mathcal{O}_{E,(p)}$-scheme $S$ to the set of isomorphism classes of tuples $(A,\lambda,i, \overline{\eta})$, where:
\begin{itemize} \item  $A$ is an abelian scheme over $S$.
 \item $\lambda: A \rightarrow A^{\vee}$ is a $\mathbb{Z}^{\times}_{(p)}$-polarization.
\item $i: \mathcal{O}_B \rightarrow \mbox{End}(A)\otimes_{\mathbb{Z}}\mathbb{Z}_{(p)}$ is a $\mathbb{Z}_{(p)}$-homomorphism satisfying $$\lambda \circ i(\alpha^{*})=i(\alpha)^{\vee} \circ \lambda$$ for all $\alpha \in \mathcal{O}_B$.
\item $\overline{\eta}$ is a prime-to-$p$ level $K^p$-structure, see the next section for details.
\end{itemize}

An isomorphism $(A,\lambda,i, \overline{\eta}) \iso (A',\lambda',i', \overline{\eta'})$ between two tuples is given by a $\mathbb{Z}^{\times}_{(p)}-$isogeny $f: A \rightarrow A'$, such that $\lambda = rf^{\vee}\circ \lambda'\circ f$ for some positive $r \in \mathbb{Z}^{\times}_{(p)}$, $f\circ i = i'\circ f$ and $\overline{\eta'}=f \circ \overline{\eta}$.

Furthermore, we assume that the determinant condition of Kottwitz is fulfilled, see \cite[\S5]{kottwitz}. Then, if the group $K^p$ is sufficiently small, the functor $\mathcal{A}$ is representable by a quasi-projective smooth scheme $\mathcal{S}\mathfrak{h}_{K^p}$ over $\mathcal{O}_{E, (p)}$, see for instance \cite[Section 2.3.3]{lan}. We fix a prime $\nu$ of $E$ over $p$ with residue field $\kappa$ and denote by $$Sh_{K^p}:= \mathcal{S}\mathfrak{h}_{K^p} \otimes \overline{\kappa}$$ the special fiber of $\mathcal{S}\mathfrak{h}_{K^p}$ over $\nu$. 

\subsection{Level $K^p$-structures and Hecke correspondences}

Let $(A,\lambda,i, \overline{\eta})$ be a tuple as above over a base scheme $S$ and $\overline{s}\in S$ a geometric point with residue field $\kappa(\overline{s})$. We denote by $A_{\overline{s}}$ the fiber of $A$ over $\overline{s}$ and by $T(A_{\overline{s}})$ its Tate module. Define $V^p(A_{\overline{s}}):= T(A_{\overline{s}}) \otimes_{\mathbb{Z}}\mathbb{A}^p_f$. A prime-to-$p$ level structure is an $\mathcal{O}_B$-linear isomorphism $$\eta:V_{\mathbb{A}^p_f} \iso V^p(A_{\overline{s}})$$ such that the following holds: the Weil pairing $e: T(A_{\overline{s}}) \times T(A^{\vee}_{\overline{s}}) \rightarrow \mathbb{A}^p_f(1)$ gives
a pairing $<\,,>_{\lambda}: V^p(A_{\overline{s}})\times V^p(A_{\overline{s}}) \rightarrow \mathbb{A}^p_f(1)$ using the polarization $\lambda:A \rightarrow A^{\vee}$, and if we identify $\mathbb{A}^p_f(1)$ with $\mathbb{A}^p_f$ we require that $\eta$ maps the pairing $<\,,>: V_{\mathbb{A}^p_f} \times V_{\mathbb{A}^p_f} \rightarrow \mathbb{A}^p_f$ to an $(\mathbb{A}^p_f)^{\times}$-multiple
of $<\,,>_{\lambda}$. In particular, for a connected scheme $S$ the algebraic fundamental group $\pi_1(S,\overline{s})$ acts continuously on $V_{\mathbb{A}_f^p}$ as symplectic similitudes, that is there is a continuous morphism $$\rho_{A_{\overline{s}}}:\pi_1(S,\overline{s}) \rightarrow \mathbf{G}(\mathbb{A}^p_f).$$ We define $\overline{\eta} := \eta \circ K^p$ and call this a prime-to-$p$ level $K^p$-structure for an arbitrary scheme $S$ if $\rho_{A_{\overline{s}}}$ factors through $K^p$ for all geometric points $\overline{s}$ of all connected components of $S$. \medskip

We now describe the prime-to-$p$ Hecke correspondences on a Shimura variety $Sh_{K^p_0}$. Let $(\mathcal{S}\mathfrak{h}_{K^p})_{K^p \subseteq \mathbf{G}(\mathbb{A}^p_f)}$ be the tower of all Shimura schemes. There is an \'etale covering $\mathcal{S}\mathfrak{h}_{K^p_1} \rightarrow \mathcal{S}\mathfrak{h}_{K^p_2}$ given by extending the prime-to-$p$ level $K^p_1$-structure to $K^p_2$ for each inclusion $K^p_1 \subseteq K^p_2$ of open compact subgroups of $\mathbf{G}(\mathbb{A}^p_f)$. The group $\mathbf{G}(\mathbb{A}^p_f)$ acts on $(\mathcal{S}\mathfrak{h}_{K^p})_{K^p \subseteq \mathbf{G}(\mathbb{A}^p_f)}$ defined by $$\mathcal{S}\mathfrak{h}_{K^p} \rightarrow \mathcal{S}\mathfrak{h}_{g^{-1}K^pg},\; (A,\lambda,i,\overline{\eta}) \mapsto (A,\lambda,i,\overline{\eta g})$$ for a $g \in \mathbf{G}(\mathbb{A}^p_f)$. The tower $(\mathcal{S}\mathfrak{h}_N)_N \rightarrow \mathcal{S}\mathfrak{h}_{K^p_0}$ over all open normal subgroups $N \subseteq K^p_0$ is a pro-\'etale Galois covering with Galois group $\varprojlim_N K^p/N = K^p$. Let $k \supseteq \kappa$ be an algebraically closed field, $x \in Sh_{K_0^p}(k)$ a point and fix a point $\tilde{x} \in (Sh_{K^p}(k))_{K^p \subseteq \mathbf{G}(\mathbb{A}^p_f)}$ over $x$. We define the prime-to-$p$ Hecke orbit $\mathcal{H}^p(x)$ of $x$ as the projection of $\mathbf{G}(\mathbb{A}^p_f)\cdot \tilde{x}$ to $Sh_{K^p_0}(k)$ and the $l$-Hecke orbit $\mathcal{H}_l(x)$ for a prime $l \neq p$ as the projection of $\mathbf{G}(\mathbb{Q}_l)\cdot \tilde{x}$ to $Sh_{K^p_0}(k)$ under the canonical morphism $\mathbf{G}(\mathbb{Q}_l) \hookrightarrow \mathbf{G}(\mathbb{A}^p_f)$. One easily sees that the Hecke correspondences are independent from the choice of the point $\tilde{x}$. Furthermore, $\mathcal{H}^p(x)$ and $\mathcal{H}_l(x)$ are countable sets due to the surjective map $\mathbf{G}(\mathbb{A}^p_f)/K^p_0 \twoheadrightarrow \mathcal{H}^p(x)$. 

\section{Monodromy groups}
Let $Z \subseteq Sh_{K^p_0}$ be a smooth, locally closed subscheme. We assume that $Z(k)$ is closed under prime-to-$p$ Hecke correspondences, that is $\mathcal{H}^p(Z(k)) \subseteq Z(k)$. Let $z \in Z$ denote the generic point of an irreducible component $Z^0$ of $Z$, which is also a connected component of $Z$. We say that the prime-to-$p$ Hecke correspondences act transitively on the set of connected components of $Z$ if $$\Pi_0(\overline{\mathcal{H}^p(Z^0(k))} \cap Z) \rightarrow \Pi_0(Z)$$ is surjective, where $\overline{\mathcal{H}^p(Z^0(k))}$ denotes the Zariski closure of the Hecke correspondences. We also consider an analogue definition for the $l$-Hecke correspondences. Furthermore, let $$\overline{\rho^p}: \pi_1(Z^0,\overline{z}) \rightarrow K^p_0 $$ respectively $$\overline{\rho_l}:\pi_1(Z^0,\overline{z}) \rightarrow K_{0,l} := \mbox{im}(K^p_0 \rightarrow \mathbf{G}(\mathbb{Q}_l))$$ denote the $K^p_0$ respectively $K_{0,l}$-conjugacy classes of representations with respect to the tuple $(A,\lambda,i,\overline{\eta})$ corresponding to the morphism $$Z^0 \hookrightarrow Sh_{K^p_0} \rightarrow  \mathcal{S}\mathfrak{h}_{K^p_0}.$$ We denote by $P$ the finite set of primes of $\mathbb{Q}$ containing $p$ and the primes $l$ such that some simple component of $\mathbf{G}_1$ is $\mathbb{Q}_l$-anisotropic. Let $\mathbb{A}^P_f$ be the finite adeles with trivial $P$ components. There is the following proposition, proven by C.-L. Chai in the Siegel case, \cite[Proposition 4.1 and 4.5.4]{chaimonodromy}:

\begin{prop} \label{monodromy}
Let $\mathcal{D}$ be an integral Shimura PEL-datum, unramified at a prime $p$. Let $\mathbf{H}_1, \dots,\mathbf{H}_n$ be the simple components of the derived group $\mathbf{G}_1$. Let $Z \subseteq Sh_{K^p_0}$ be a smooth, locally closed subscheme. \begin{enumerate} \item Suppose that for a prime $l \notin P$ the $l$-Hecke correspondences of the simply connected covering of $\mathbf{G}_1$ act transitively on the set of connected components of $Z$. If for all $i=1,\dots,n$ the set $\mbox{im}(\rho_l) \cap \mathbf{H}_i(\mathbb{Q}_l)$ is not finite, then $\mbox{im}(\rho_l) = K_{0,l}$. \item Suppose that the prime-to-$P$ Hecke correspondences of the simply connected covering of $\mathbf{G}_1$ act transitively on the set of connected components of $Z$. If for all $i=1,\dots,n$ and all $l \notin P$ the set $\mbox{im}(\rho_l) \cap \mathbf{H}_i(\mathbb{Q}_l)$ is not finite, then $\mbox{im}(\rho^P) = K^P_0$. \end{enumerate}

\end{prop}

\begin{proof}
Of course (1) implies (2). So let us show (1).
Choose a prime $l \notin P$. Let $N \subseteq K^p_0$ denote an open normal subgroup such that $K^p_0/N$ is trivial outside of $l$, and consider the pro-\'etale covering $$(Sh_N)_{N \subseteq \mathbf{G}(\mathbf{A}^p_f)} \rightarrow Sh_{K^p_0}.$$ The group $K_{0,l}$ acts continuously on this covering via $l$-Hecke correspondences. Consider the pro-\'etale covering of $Z^0$ obtained by base change with $Z^0 \hookrightarrow Sh_{K^p_0}$,  $$\widetilde{Z^0}:=(Sh_N\times_{Sh_{K^p_0}} Z^0)_{N \subseteq \mathbf{G}(\mathbf{A}^p_f)}  \rightarrow Z^0,$$ together with the induced action of $K_{0,l}$. We choose a point $\tilde{z} \in \widetilde{Z^0}$ over $\overline{z}$ and thus obtain a connected component $\widetilde{Z^0}^0$. The algebraic fundamental group of $Z^0$ acts on $\widetilde{Z^0}$ through the obvious morphism $$\pi_1(Z^0,\overline{z}) \rightarrow \pi_1(Sh^0_{K^p_0}, \overline{z}) \rightarrow \Aut_{Sh_{K^p_0}}((Sh_N)_{N \subseteq \mathbf{G}(\mathbf{A}^p_f)}) = K^p_0,$$ where $Sh^0_{K^p_0}$ is the connected component containing $Z^0$. One easily sees that this morphism coincides with the morphism $$\rho^p:\pi_1(Z^0,\overline{z}) \rightarrow  K^p_0$$ coming from the prime-to-$p$ level $K^p_0$-structure. It follows from the general theory of the fundamental group that $M:=\mbox{im}(\rho_l)$ is the stabilizer of the connected component $\widetilde{Z^0}^0$ with respect to the $K_{0,l}$-action. We get a homeomorphism $K_{0,l}/M \iso \Pi_0(\widetilde{Z^0})$ of profinite sets. For the pro-\'etale covering $$\widetilde{Z}:=(Sh_N\times_{Sh_{K^p_0}} Z)_{N \subseteq \mathbf{G}(\mathbf{A}^p_f)}  \rightarrow Z$$ we get an analogous continuous bijection $$\mathbf{G}_1(\mathbb{Q}_l)/\mbox{Stab}_{\mathbf{G}_1}(\widetilde{Z^0}^0) \iso \Pi_0(\widetilde{Z})$$ since the $l$-Hecke correspondences act transitively on the connected components of $Z$ and hence also on the connected components of $\widetilde{Z}$. It follows with the same arguments as in \cite[Lemma 2.8]{chaimonodromy} that this isomorphism is a homeomorphism. 

\medskip
As explained in the proof of \cite[Proposition 4.1]{chaimonodromy}, the algebraic group $$\mathbf{M} := \overline{\mbox{im}(\rho_l)}^0 \subseteq \mathbf{G}_{\mathbb{Q}_l}$$ is semisimple, hence it lies in $\mathbf{G}_{1, \mathbb{Q}_l}$, and we have $$\mbox{Stab}_{\mathbf{G}_1}(\widetilde{Z^0}^0) \subset N_{\mathbf{G}_1}(\mathbf{M})(\mathbb{Q}_l).$$ Since there is a continuous surjection $$\mathbf{G}_1(\mathbb{Q}_l)/\mbox{Stab}_{\mathbf{G}_1}(\widetilde{Z^0}^0) \twoheadrightarrow \mathbf{G}_1(\mathbb{Q}_l)/N_{\mathbf{G}_1}(\mathbf{M})(\mathbb{Q}_l)$$ and the set on the left is profinite, the quotient $\mathbf{G}_1(\mathbb{Q}_l)/\mathbf{N}(\mathbb{Q}_l)$ is compact, where we set $$\mathbf{N}:= N_{\mathbf{G}_1}(\mathbf{M})^0.$$ Then \cite[Proposition 9.3]{boreltits} implies that $\mathbf{N}$ contains a maximal $\mathbb{Q}_l$-split solvable subgroup $A$ of $\mathbf{G}_{1, \mathbb{Q}_l}$, and because $\mathbf{G}_{1, \mathbb{Q}_l}$ is isotropic, \cite[Proposition 8.4,8.5]{boreltits} shows that $A \neq 0$ is the unipotent radical of a minimal parabolic subgroup of $\mathbf{G}_{1, \mathbb{Q}_l}$. Finally, we use \cite[Proposition 8.6]{boreltits} and conclude that $\mathbf{N}$ contains the smallest normal subgroup of $N_{\mathbf{G_1}}(\mbox{rad}_u(\mathbf{N}))$ containing $A$. But $\mathbf{N}$ is reductive since $\mathbf{G}_1$ and $\mathbf{M}$ are semisimple (see \cite[Lemma 3.3]{chaimonodromy}), hence we showed that it contains a nontrivial normal connected subgroup of $\mathbf{G}_{1, \mathbb{Q}_l}$. The assumptions on $\rho_l$ then yield $\mathbf{N} = \mathbf{G}_{1, \mathbb{Q}_l}$, so $\mathbf{M}$ itself is a nontrivial normal subgroup of $\mathbf{G}_{1, \mathbb{Q}_l}$ intersecting all simple subgroups. This implies $\mathbf{M} = \mathbf{G}_{1, \mathbb{Q}_l}$.
\medskip

We now know that the Lie algebras of the $l$-adic Lie group $M$ and its Zariski closure coincide, because $\mathbf{M}$ is semisimple (see \cite[Corollary 7.9]{borellinear}), so $M \subseteq K_{0,l}$ contains an open subgroup. It follows that $$K_{0,l}/M \iso \Pi_0(\widetilde{Z^0})$$ is finite. Because $Z$ is quasi-projective, it has only finitely many connected components, which implies that $\Pi_0(\widetilde{Z})$ is also finite. If $\#\Pi_0(\widetilde{Z}) \neq 1$, then the simply connected group $\mathbf{G}^{sc}_1(\mathbb{Q}_l) = \mathbf{H}_1^{sc}(\mathbb{Q}_l)\times \dots \times \mathbf{H}_n^{sc}(\mathbb{Q}_l)$ contains a nontrivial subgroup of finite index. But the Kneser-Tits conjecture for simple and simply connected $\mathbb{Q}_l$-isotropic groups implies that all the groups $\mathbf{H}^{sc}_i(\mathbb{Q}_l)$ have no nontrivial noncentral normal subgroups, see \cite[Theorem 7.1, 7.6]{platonovrapinchuk}. It follows that they do not contain any nontrivial subgroup of finite index, and the same holds for $\mathbf{G}_1^{sc}(\mathbb{Q}_l)$. This shows $$K_{0,l}/M \iso \Pi_0(\widetilde{Z^0}) = \{1\}.$$ 
\end{proof}

\begin{cor} \label{connected}
 If the conditions of Proposition \ref{monodromy} are fulfilled, then $Z$ is connected.
\end{cor}
\begin{proof}
 This follows at once from the proof above.
\end{proof}

\section{Hecke correspondences and Newton strata}

From now on, we assume that the Shimura PEL-datum $\mathcal{D}$ is of type $A$ or $C$, which implies that the group $\mathbf{G}$ is connected. Let $x \in Sh_{K^p_0}(k)$ be a geometric point, given by a tuple $(A,\lambda,i,\overline{\eta})$. We define a functor by $$\mathbf{I}_x(R):=\{g \in \mbox{End}_B(A) \otimes_{\mathbb{Z}_{(p)}}R \;|\; gg^* \in R^{\times}\},$$ where $R$ is a $\mathbb{Z}_{(p)}$-algebra and $^*$ is the Rosati involution with respect to the polarization $\lambda$. This defines a reductive group over $\mathbb{Q}$. We fix an element in the prime-to-$p$ level $K_0^p$-structure $\overline{\eta}$, which gives together with the injection $$\mbox{End}(A)\otimes_{\mathbb{Z}}\mathbb{A}_f^p \hookrightarrow \mbox{End}(V^p(A))$$ proved by J. Tate an injection $$\mathbf{I}_{x, \mathbb{A}^p_f} \hookrightarrow \mathbf{G}_{\mathbb{A}^p_f}$$ of algebraic groups. One easily sees that the tuple $(A,\lambda,i,\overline{\eta})$ is isomorphic to $(A,\lambda,i,\overline{\eta g})$ if and only if $g$ lies in the set  $\mathbf{I}_x(\mathbb{Z}_{(p)})K^p_0$, where we consider $\mathbf{I}_x(\mathbb{Z}_{(p)}) \subset \mathbf{G}(\mathbb{A}^p_f)$ with respect to the chosen element of the level structure. If $g \in \mathbf{G}(\mathbb{Q}_l) \subset \mathbf{G}(\mathbb{A}^p_f)$, then the above tuples are isomorphic if and only if $$g \in \mathbf{G}(\mathbb{Q}_l) \cap \mathbf{I}_x(\mathbb{Z}_{(p)})K^p_0 = \mathbf{I}_{x,l}(\mathbb{Z}_{(p)})K_{0,l},$$ where $\mathbf{I}_{x,l}(\mathbb{Z}_{(p)})$ denotes the group of $\mathbb{Z}_{(p)}$-isogenies in $\mathbf{I}_x(\mathbb{Z}_{(p)}) \subset \mathbf{G}(\mathbb{A}^p_f)$ whose $l'$-components lie in $K_{0,l'}$ for all primes $l' \neq l$. It follows that there are bijections \begin{eqnarray*} \mathbf{I}_x(\mathbb{Z}_{(p)})\backslash \mathbf{G}(\mathbb{A}^p_f)/K^p_0 \iso \mathcal{H}^p(x) \\ \mathbf{I}_{x,l}(\mathbb{Z}_{(p)})\backslash \mathbf{G}(\mathbb{Q}_l)/K_{0,l} \iso \mathcal{H}_l(x) \end{eqnarray*}

We are interested in points with infinite Hecke orbit. It turns out that $\mathcal{H}^p(x)$ is infinite if $x$ does not lie in the basic locus of the Newton stratification of $Sh_{K_0^p}$. For an overview of the Newton stratification see for instance \cite[7.2]{wedhornviehmann} and the references therein. Here we give a short description of the properties that are important for the purpose of this paper. \medskip

Let $L$ denote the quotient field of the Witt ring of $k$ with Frobenius morphism $\sigma$. Let $B(\mathbf{G})$ be the set of $\mathbf{G}(L)-\sigma$-conjugacy classes $\{g^{-1}b\sigma(g) \;|\; g \in \mathbf{G}(L)\}$  of elements $b \in \mathbf{G}(L)$. There is a partial ordering on $B(\mathbf{G})$ and a finite subset $B(\mathbf{G},\mu) \subset B(\mathbf{G})$ containing a unique maximal element $b_{\mu}$ and a unique minimal element $b_0$, see \cite[\S6]{kottwitziso2}. Here $\mu$ denotes a dominant cocharacter of a maximal torus of $\mathbf{G}$ defined by the Shimura PEL-datum $\mathcal{D}$. A $\sigma$-conjugacy class $\overline{b}$ is called \emph{basic} if it contains an element lying in $\mathbf{T}(L)$, where $\mathbf{T}$ is an elliptic maximal torus of $\mathbf{G}$. There is a stratification $$Sh_{K^p_0} = \bigcup_{b \in B(\mathbf{G}, \mu)}\mathcal{N}_b$$ with locally closed subsets $\mathcal{N}_b \subset Sh_{K_0^p}$ such that $\overline{\mathcal{N}_b} \subseteq \bigcup _{b \leq b'}\mathcal{N}_{b'}$. For each $b \in \mathbf{G}(L)$ there is an affine algebraic $\mathbb{Q}_p$-group $\mathbf{J}_b$ defined by $$\mathbf{J}_b(R):= \{g \in \mathbf{G}(R \otimes_{\mathbb{Q}_p}L)\;|\; g(b\sigma)=(b\sigma)g\},$$ see \cite[Proposition 1.12]{rapoportzink}. Observe that there is an injection $\mathbf{I}_{x,\mathbb{Q}_p} \hookrightarrow \mathbf{J}_b$ of algebraic groups. Recall that a geometric point $x = (A,\lambda, i, \overline{\eta})$ is called \emph{hypersymmetric} if $\mathbf{I}_{x, \mathbb{Q}_p} \iso \mathbf{J}_b$. We have the following lemma, which was observed by C.-L. Chai in the Siegel case in \cite[Proposition 1]{chaiordinary}.       

\begin{lemma} \label{heckeorbit}
 Let $x \in Sh_{K^p_0}(k)$ be a geometric point. \begin{enumerate} \item If $x \nin \mathcal{N}_{b_0}(k)$, then $\mathcal{H}_l(x)$ is not finite for all primes $l$ such that $\mathbf{G}_{\mathbb{Q}_l}$ is split. In particular, $\mathcal{H}^p(x)$ is not finite. \item If $x \in \mathcal{N}_{b_0}(k)$, then $\mathcal{H}^p(x)$ is finite. \end{enumerate}
\end{lemma}

\begin{proof}
It follows from \cite[Proposition 2.4]{rapoportrichartz} that $b_0$ is the only basic point in $B(\mathbf{G},\mu)$. Let us consider the first statement of the lemma. It suffices to show that $\mathbf{I}_x(\mathbb{Q}_l)\backslash \mathbf{G}(\mathbb{Q}_l)$ is not compact, which is equivalent to $\mathbf{I}_x(\mathbb{Q}_l)$ not containing a maximal $\mathbb{Q}_l$-split connected solvable subgroup, which is a Borel subgroup since $\mathbf{G}_{\mathbb{Q}_l}$ is split. It follows that this quotient is compact if and only if $\mathbf{I}_{x,\mathbb{Q}_l}$ is parabolic, but the only parabolic and reductive subgroup is $\mathbf{G}_{\mathbb{Q}_l}$ itself. Furthermore, \cite[Corollary 1.14, Remark 1.15]{rapoportzink} implies that the reductive group $\mathbf{J}_b$ is a form of $\mathbf{G}_{\mathbb{Q}_p}$ if and only if $b$ is basic.  Then the first statement of the Lemma holds because $\mathbf{I}_x$ cannot be a form of $\mathbf{G}$.  \medskip

A geometric point $x$ in the basic locus of $Sh_{K^p_0}$ is hypersymmetric, see \cite[Theorem 6.30]{rapoportzink}, hence  $\mathbf{I}_{x, \mathbb{Q}_p} \iso \mathbf{J}_b$ and  $\mathbf{I}_{x, \mathbb{A}^p_f} \iso \mathbf{G}_{\mathbb{A}^p_f}$. Then $(2)$ follows at once from the finiteness of class numbers of algebraic groups, see for instance \cite[Theorem 5.1]{platonovrapinchuk}.      
\end{proof}

We are now ready to prove the main result of this paper. We consider a smooth, locally closed subscheme $Z \subseteq Sh_{K^p_0}$ as above and fix a connected component $Z^0$ with generic point $z$.

\begin{thm}
Suppose that the prime-to-$P$ Hecke correspondences of $\mathbf{G}_1$ act transitively on the set of connected components of $Z$. If $z \nin \mathcal{N}_{b_0}$, then $Z$ is connected.  
\end{thm}

\begin{proof}
It suffices to show that $\mbox{im}(\rho_l)$ is not finite for all $l \neq p$. Then Corollary \ref{connected} implies that $Z$ is connected, because the image of $\rho_l$ lies in the simply connected semisimple group $\mathbf{G}_1$. So suppose $\mbox{im}(\rho_l)$ is a finite group for some prime $l \neq p$ and let $(A,\lambda,i, \overline{\eta})|_{Z^0}$ correspond to the morphism $Z^0 \hookrightarrow Sh_{K^p_0}$. Then there is a finite field extension $L$ of $\kappa(z)$ such that for the normalization $Z'$ of $Z^0$ in $L$ the abelian scheme $A':= A \times_{Z^0}Z'$ has the trivial monodromy representation $\rho'_l$. A successive application of the result \cite[Theorem 2.1]{oortmoduli} of F. Oort then implies that $A'_L$ is isogenous to an abelian variety $A_0$ defined over $\overline{\mathbb{F}}_p$, which extends to an isogeny $A' \rightarrow A_0 \times_{\overline{\mathbb{F}}_p}Z'$ using \cite[Proposition I.2.7]{faltingschai}. \medskip

We claim that the Zariski closure of $Z^0$ in $Sh_{K^p_0}$ is a proper scheme over $\overline{\mathbb{F}}_p$. To show this, consider the Zariski closure $$\overline{Z^0} \subseteq \overline{Sh}_{K^p_0}$$ in the toroidal compactification  $\overline{Sh}_{K^p_0}$ of the Shimura variety, which is a proper and smooth scheme over $\overline{\mathbb{F}}_p$ containing $Sh_{K^p_0}$ as an open dense subscheme, see \cite[Theorem 6.4.1.1]{lan}. Furthermore, there is a degenerating family $(G,\lambda_G,i_G,\overline{\eta})$ in the sense of \cite[Definition 5.4.2.1]{lan}, where $G$ is a semiabelian scheme over $\overline{Sh}_{K^p_0}$ with PEL-structure that restricts to the universal abelian scheme with PEL-structure over $Sh_{K^p_0}$. Denote by $$\overline{Z'} \rightarrow \overline{Z^0}$$ the normalization of $\overline{Z^0}$ in $L$. Since $Z' \subseteq \overline{Z'}$ is an open subscheme, we can again use \cite[Proposition I.2.7]{faltingschai} and extend the isogeny $A_0 \times_{\overline{\mathbb{F}}_p}Z' \rightarrow A'$ uniquely to a morphism $$A_0 \times_{\overline{\mathbb{F}}_p}\overline{Z'} \rightarrow G_{\overline{Z'}}$$ of semiabelian schemes. It is easy to see that then $G_{\overline{Z'}}$ is an abelian scheme over $\overline{Z'}$. This implies that the image of the morphism $\overline{Z'} \rightarrow \overline{Sh}_{K^p_0}$ lies in the open subset where $G$ is an abelian variety, hence $\overline{Z^0}$ lies in $Sh_{K^p_0}$, which shows the claim. \medskip

In \cite{wedhornviehmann}, E. Viehmann and T. Wedhorn generalize the Ekedahl-Oort stratification to special fibers of PEL-Shimura varieties of type $A$ and $C$. That is, there is a finite and partially ordered set $(^JW, \preceq)$ containing a unique minimal and maximal element, such that there is a stratification $$Sh_{K_0^p} = \bigcup_{\omega \in ^JW}S^{\omega}$$ of locally closed and quasi-affine subschemes $S^{\omega}$, see \cite[\S2,\S4 and Theorem 9.6]{wedhornviehmann}. Furthermore, $$\overline{S^{\omega}} = \bigcup_{\omega' \preceq \omega}S^{\omega'}$$ see \cite[Theorem 6.1]{wedhornviehmann}, and the minimal stratum is closed, nonempty, and lies in $\mathcal{N}_{b_0}$ \cite[Proposition 8.16]{wedhornviehmann}.  \medskip

We know that $\overline{Z^0}$ is a proper scheme over $\overline{\mathbb{F}}_p$ and that it is closed under $\mathcal{H}^p$ since $Z^0$ is closed under Hecke correspondences. It follows directly from the definition of the Ekedahl-Oort stratification that each stratum is also closed under $\mathcal{H}^p$. This implies together with Lemma \ref{heckeorbit} that $\overline{Z^0} \cap \mathcal{N}_{b_0} \neq \emptyset$. For if there is a point of $\overline{Z^0}$ not lying in the basic Newton stratum, it cannot lie in the smallest Ekedahl-Oort stratum. Furthermore, the closure of its prime-to-$p$ Hecke orbit is proper and of dimension $>0$, so it has to meet a smaller Ekedahl-Oort stratum. Then repeat this argument. But this statement is a contradiction to the isogeny $A' \times_{Z'} \overline{Z'} \rightarrow A_0 \times_{\overline{\mathbb{F}}_p}\overline{Z'}$, which says that $\overline{Z^0}$ lies in a single Newton stratum.
\end{proof}

\bibliographystyle{amsplain}
\providecommand{\bysame}{\leavevmode\hbox to3em{\hrulefill}\thinspace}
\providecommand{\MR}{\relax\ifhmode\unskip\space\fi MR }
\providecommand{\MRhref}[2]{%
  \href{http://www.ams.org/mathscinet-getitem?mr=#1}{#2}
}
\providecommand{\href}[2]{#2}

\end{document}